\documentclass[12pt]{article}

\input{mathdefs.sty}

\usepackage[dvips, outer=0.95in, top=1in, bottom=1.2in]{geometry}

\usepackage{bm}

\newcommand{\sns}[1]{\ensuremath{\text{\upshape{\sffamily{#1}}}}}

\title{Limiting distribution of visits of sereval rotations to shrinking intervals}
\author{Ilya Vinogradov}
\date{\today}

\newtheorem{theorem}{Theorem}
\newtheorem{lemma}[theorem]{Lemma}
\newtheorem{prop}[theorem]{Proposition}

\newcounter{theremark}
\newenvironment{remark}{\medskip\parindent=0pt\small\textbf{Remark \arabic{theremark}\addtocounter{theremark}{1}.\ }\ignorespaces}{\medskip\par\normalsize\parindent=18pt}

\begin{document}

\maketitle

\begin{abstract}
We show that given $n$ normalized intervals on the
unit circle, the numbers of visits of $d$ random rotations to these intervals have
a joint limiting distribution as lengths of trajectories tend to infinity. If $d$ then tends to infinity, then the numbers of points in different intervals become asymptotically independent unless an arithmetic obstruction arises. This is a generalization of earlier results of J. Marklof.   
\end{abstract}

The following question arises from two results of Marklof about gap distribution for rotations. Fix a point $\xi$ in $[0,1)$ and let $B_N=(0,N^{1/(d-1)}]^{\oplus{(d-1)}}\oplus\R\subset \R^d$. What what is the limiting behavior of the number of points of the form $$\left\{\sum_{i=1}^{d-1} m_i\alpha_i\pmod 1\colon m_i\in[1,N^{1/(d-1)}]\cap\mathbf Z, 1\le i\le d-1\right\},$$ for $\alpha_i\in[0,1)$ that land in $(\xi-\frac\sigma N,\xi+\frac\sigma N)$ as $N\to\infty$? In \cite{marklof_2000}, J. Marklof showed  that 
$$\mathrm{leb}\left\{(\alpha,\xi)\in[0,1)^{d-1}\times[0,1)\colon\#\{\bm m\in B_N\cap\Z^d\colon \xi+\sum_{i=1}^{d-1}m_i\alpha_i+m_d\in (-\tfrac{\sigma}{N},\tfrac{\sigma}{N})\}=A\right\}\to P^{(d)}(A)$$ 
as $N\to\infty$ and found its decay as $A\to\infty$. His main tool was the mixing property of a diagonal flow on $\sldr/\sldz$ that had been proved by Moore \cite{moore_ergodicity_1966}. In a later note Marklof remarked that for one variable (that is, $d=2$), a stronger result is true due to a Theorem of Shah \cite{shah_limit_1996}. Namely, for fixed $\xi\in[0,1)\setminus \Q$, $$\mathrm{leb}\left\{\alpha\in[0,1)\colon\#\{ m\in \{1,\dots,N\}\colon \xi+m\alpha\pmod 1\in (-\tfrac{\sigma}{N},\tfrac{\sigma}{N})\}=A\right\}\to P^{(2)}(A).$$ This result uses Ratner's Theorem on measures invariant under unipotent flows \cite{ratner}. We will generalize the theorems mentioned above to joint limiting probability distributions for  several intervals and study their large $d$ limits.

\section{Notation and results}

We will use the following
notation. 
\begin{itemize}\item $N$, $n$, and $d$ denote positive integers with $d\ge 2$;
\item upper indices (usually $j$) run from 1 to $n$ and lower indices (usually $i$) run from 1 to $d$ unless stated otherwise;
\item $\bm m=(m_1,\dots,m_d)$ is a vector of $d$ integers;
 \item $\bm\sigma=(\sigma^1,\dots,\sigma^n)$ is a positive vector ($\sigma^j>0$
for all $j$);
\item $\bm\alpha=(\alpha_1,\dots,\alpha_{d-1})\in(\R/\Z)^{d-1}$;
\item $\bm\xi=(\xi^1,\dots,\xi^n)\in(\R/\Z)^n$;
\item $\bm\tau=(\tau^1,\dots,\tau^n)$ is a real vector;
\item $\Pois\sigma$ denotes the Poisson distribution with parameter $\sigma$. 
\end{itemize}
If $n=1$, we will write $\sigma$ instead of $\sigma^1$ and $\bm\sigma$ and
similarly for other variables. Let $B_N=(0,N^{1/(d-1)}]^{\oplus (d-1)}\oplus\R\subset \R^d$ as before. For a measurable set $S$ define random variables $X_{\xi,S}^{N,d}\colon [0,1)\to\Z$ by $$X_{\xi,S}^{N,d}=\#\left\{\bm m\in\Z^d\cap B_N\colon
\sum_{i=1}^{d-1}m_i\alpha_i+m_d\in \xi+\frac{S}{N}\right\}.$$ We will usually suppress the upper indices on $X_{\xi,S}$. For a vector $\bm\xi\in \T^n$, the set $\Xi\subset\T^n$ is the  closure of the orbit of rotation by $\bm\xi$ on the torus: $\Xi=\overline{\{k\bm\xi\colon k\in\Z\}}$; it is the smallest closed Lie subgroup of $\T^n$ that contains $\bm\xi$. 

Our results for limiting distributions of $X_{\xi,S}$ are as follows. 

\begin{theorem}
%Fix $n$, $\bm\xi$, and $\bm\sigma$.
%Then, $$\mathrm{leb}\left\{\alpha\colon 
%\begin{array}{c}
%\#\left\{m\colon
%\xi_1+\sum_{i=1}^{d-1}m_i\alpha_i+m_d\in (0,\tfrac{\sigma_1}{N})\right\}=A_1\\
%\vdots\\
%\#\left\{m\colon 
%\xi_n+\sum_{i=1}^{d-1}m_i\alpha_i+m_d\in (0,\tfrac{\sigma_n}{N})\right\}=A_n
%\end{array}\right\}\to \P_{n,\bm\sigma, \Xi}^{(d)}(A_1,\dots,A_n)$$ as
%$N\to\infty$, where $\Xi=\overline{\Z\bm\xi}\subset\T^n$.
Fix any absolutely continuous probability measure on $[0,1)$. With notation as above, the distribution of $$\bm X_{\bm \xi,\bm \tau,\bm\sigma}=(X_{\xi^1,(\tau^1,\tau^1+\sigma^1)},\dots,X_{\xi^n,(\tau^n,\tau^n+\sigma^n)})$$ has a weak limit  as $N\to\infty$; we denote it by $\P_{n,\bm\sigma, \Xi,\bm\tau/\Xi}^{(d)} $. 
\end{theorem}

In other words, the numbers of points in shrinking segments
$(\xi^j+\frac{\tau^j}{N},\xi^j+\frac{\sigma^j+\tau^j}{N})$, $1\le j\le n$, with fixed ``centers'' $\xi^j$ have a joint
limiting distribution as $N$ tends to infinity. The limiting distribution depends on $d$, $n$, $\bm\sigma$, $\Xi$, and  $\bm\tau$ modulo $\Xi$. In particular, if $\Xi=\T^n$, then the distribution is independent of $\bm\tau$.

\begin{remark}
Jens Marklof proved special cases of this theorem. He proved the case $n=1$,
$d=2$ in \cite{marklof_survey} and the case $n=1$ and arbitrary $d$ with average over $\xi$ in
\cite{marklof_2000}.
\end{remark}

\begin{theorem}
Let $\P^{(d)}_{n,\bm\sigma,\Xi,\bm\tau/\Xi}$ be the distribution from Theorem 1. Then, $\P^{(d)}_{n,\bm\sigma,\Xi,\bm\tau/\Xi}$ has a weak limit as $d\to\infty$. Furthermore, $$\P^{(d)}_{n,\bm\sigma,\Xi,\bm\tau/\Xi}\Longrightarrow (\Pois \sigma_1,\dots,\Pois \sigma_n)$$ as $d\to\infty$ iff $$(\tau^j,\tau^j+\sigma^j)\cap (\tau^{j'},\tau^{j'}+\sigma^{j'})=\emptyset\text{ whenever }\xi^j=\xi^{j'}.$$
\end{theorem}

In effect, this Theorem says that as the number of rotations tends to infinity, the gap lengths exhibit random behavior. However, for every finite $d$ and $\Xi$ (with $n\ge2$) we have that $\P_{n,\bm\sigma, \Xi,\bm\tau/\Xi}^{(d)}$ is dependent.

%\begin{remark}\label{rm:general}
%This theorem has generalizations that are simple but more cumbersome to state.
%Firstly, any absolutely continuous measure on $\alpha$ can be taken instead of
%the Lebesgue measure without changing the final limiting distribution
%$\P_{n,\bm\sigma, \Xi}^{(d)}$. Secondly, we can take a more general box
%$\prod_{j=1}^n(\tau_j/N,(\tau_j+\sigma_j)/N)$ and obtain a possibly different limiting
%distribution. Whether the limiting distribution stays the same depends on
%the interplay between $\Xi$ and $\bm R$. For example, if $\Xi=\T^n$ (the
%independent case; almost all $\bm \xi$ satisfy this), then $\bm \tau$ can be any
%real vector. If $\Xi=\{(\xi,\dots,\xi)\}$, then we can ensure that $\P_{n,\bm\sigma,
%\Xi}^{(d)}$ doesn't change if $\bm \tau=(\tau,\dots,\tau)$. More examples can be easily
%constructed. 
%\end{remark}

\bigskip

\textbf{Acknowledgements. }I would like to thank Ya. Sinai, E. Lindenstrauss, J. Marklof, F. Cellarosi, A. Salehi Golsefidy, and Z. Wang for helpful discussions.

\section{Large $N$ limit}

\begin{proof}[Proof of Theorem 1]
%Define the homogeneous space $\sldr \ltimes(\R^d)^{\oplus n}/\sldz\ltimes(\Z^d)^{\oplus n}=L/\Lambda$  multiplication $$(M,v_1,\dots,v_n)(N,w_1,\dots,w_n)=(MN,v_1+Mw_1,\dots ,v_n+Mw_n)$$ in $L$. Define $L'\subset L$ to be the closure of $$(M,M(a_1+Nv_1),\dots,M(a_n+Nv_n))$$ with $M\in\sldr$, $N\in\sldz$, $ a_j\in \Z^d$, $v_j=(0,\dots,\xi_j)^T$. It is clear that $\Lambda'=\Lambda\cap L'$ is a lattice in $L'$ and we can form the homogeneous space $L'/\Lambda'$. 
We reformulate the problem in the language of homogeneous spaces. Let $L=\sldr \ltimes(\R^d)^{\oplus n}$ and let $\Lambda=\sldz\ltimes(\Z^d)^{\oplus n}\subset L$. Multiplication law on $L$ is given by $$(\sns M,\bm v^1,\dots,\bm v^n)(\sns N,\bm w^1,\dots,\bm w^n)=(\sns M\sns N,\bm v^1+\sns M\bm w^1,\dots ,\bm v^n+\sns M\bm w^n).$$ It is well-known that $\Lambda\subset L$ is a non-cocompact lattice. The homogeneous space $L/\Lambda$ is a bundle over $\sldr/\sldz$ with fiber $(\T^d)^{\oplus n}$.

Given a set of vectors $\bm v^1,\dots,\bm v^n\in\T^d$, let $$L_V=\{(1,\bm v^1,\dots,\bm v^n)^{-1}(\sns M,0,\dots,0)(1,\bm v^1\dots,\bm v^n)\mid \sns M\in\sldr\}\subset L;$$ it is of course isomorphic to \sldr. Also define $\hat L_V$ to be the smallest group containing $L_V$ that is defined over \Q. Dimension of $\hat L_V$ depends on $\bm v^j$.  If all vectors $\bm v^j$ have rational coordinates, then $\hat L_V= L_V$. Otherwise the fiber over the identity in $\hat L_V$ is the smallest \Q-vector space containing the identity fiber for $L_V$. This construction can be  carried to other points. 
%$$\hat L_V=\bigcap_{\substack{L\supset L'_V\supset L_V\\ L'_V\cap(1,(\R^d)^{\oplus n})\text{ is linear defined over }\Q}}L'_V.$$ 
Finally set $\hat\Lambda_V=\hat L_V\cap\Lambda$ which is a lattice in $\hat L_V$ by construction. 

For our purposes fix $\bm v^j=(0,\dots,0,\xi^j)^T$. For this choice of $\bm v^j$ we get the homogeneous space $\hat L_V/\hat\Lambda_V$. We have constructed $\hat L_V$ so that $\pi(\hat L_V)=\overline{\pi(L_V)}$, where $\pi\colon L\to L/\Lambda$ is the canonical projection. The structure of this space depends on  $\Xi=\overline{\Z\bm\xi}\subset\T^n.$ It is a subbundle of $L/\Lambda$: the base is still $\sldr/\sldz$ but the fiber is $\Xi^d$ after reordering coordinates. 

We define $f_{\bm\tau,\bm\sigma}\colon \hat L_V/\hat \Lambda_V\to\Rn$ by
$$f_{\bm\tau,\bm\sigma}(\sns M,\bm v^1,\dots,\bm v^n)=(g_{\tau^1,\sigma^1}(\sns M,\bm v^1),\dots,g_{\tau^n,\sigma^n}(\sns M,\bm v^n)),$$ where $$g_{\tau,\sigma}(\sns M,\bm v)=\sum_{\bm m\in \Z^d\setminus\{0\}} \chi_1(\tilde{ m}_1)\dots\chi_1(\tilde{m}_{d-1})\chi_{(\tau,\tau+\sigma)}(\tilde{m}_d),$$ $$\chi_\sigma(x)=
\begin{cases}1&x\in(0,\sigma)\\
0&\text{otherwise,}
\end{cases}$$ and $$\tilde{\bm m}=\sns M\bm m+\bm v.$$ It is easily seen that $f$ is $\hat\Lambda_V$ invariant and hence well-defined on the quotient. 

%In the special case that $$M=
%\begin{pmatrix}
%N^{-1/(d-1)}\\&\ddots\\&&N^{-1/(d-1)}\\&&&N
%\end{pmatrix}
%\cdot
%\begin{pmatrix}
% 1\\
%&\ddots\\
%&&1\\
%\alpha_1&\dots&\alpha_{d-1}&1
%\end{pmatrix}
%$$ and $v_j=(0,\dots,\xi_j)^T$, $f(M,v_1,\dots,v_n)$ is the vector giving the number of points in each segment for the specified values of $\bm\alpha$, $N$, $\bm\xi$. 

We need to show that $\leb\{f_{\bm\tau,\bm\sigma}=(A^1,\dots,A^n)\}\to\P^{(d)}_{n,\bm\sigma,\Xi,\bm\tau/\Xi}(A^1,\dots,A^n)$ as $N\to\infty$. To this end we use a theorem of Shah (Theorem 1.4 in \cite{shah_limit_1996}). The form we need is the following:

\begin{theorem}[Shah]\label{th:shah}
Let $$U_{\bm\alpha}=\begin{pmatrix}
 1\\
&\ddots\\
&&1\\
\alpha_1&\dots&\alpha_{d-1}&1
\end{pmatrix}\quad\text{and}\quad
\Phi_t=\begin{pmatrix}
e^{-t}\\
&\ddots\\
&&e^{-t}\\
&&&e^{(d-1)t}
\end{pmatrix}.$$ Let $L$ be a Lie group, $\Lambda\subset L$ a lattice, $\phi\colon \sldr\to L$ an embedding. If the image of $\phi$ is dense when projected to $L/\Lambda$, then for any bounded continuous $\eta$ \beq\label{eq:shah}\lim_{t\to\infty}\int_{\R^{d-1}}\eta(\phi(\Phi_t U_{\bm\alpha}))d\nu(\bm\alpha)=\int_{L/\Lambda} \eta(\sns M)d\mu(\sns M),\eeq where $\nu$ is any absolutely continuous probability measure on $U_{\bm\alpha}$ and  $\mu$ is the Haar probability measure on $L/\Lambda$.
\end{theorem}

\begin{remark}
In effect, the Theorem says that the unstable manifold $U_{\bm\alpha}$ is equidistributed in the larger homogeneous space $L/\Lambda$ provided the density assumption is satisfied. 
\end{remark}

We set $N=e^{(d-1)t}$ and apply the Theorem \ref{th:shah} with $L=\hat L_V$, $\Lambda=\hat\Lambda_V$, and $$\phi\colon \sns M\mapsto (1,\bm v^1,\dots,\bm v^n)^{-1}(\sns M,0,\dots,0)(1,\bm v^1,\dots,\bm v^n).$$ This ensures density after projecting to $\hat L_V/\hat\Lambda_V$ by construction. We now use the following elementary Lemma to construct appropriate functions $\eta$.

\begin{lemma}
Let $\N_0=\N\cup\{0\}$. Then $p_2\colon \N_0\times\N_0\to \N_0$ given by $$(x,y)\mapsto \binom{x+y+2}{2}-(y+1)$$ is a bijection. 
\end{lemma}
By induction, there exists a polynomial bijection between $\N_0^n$ and $\N_0$ for each $n$; call it $p_n$. Set $$h_n(\sns M,\bm x^1,\dots,\bm x^n)=p_n(g_{\tau^1,\sigma^1}(\sns M,\bm x^1),\dots,g_{\tau^n,\sigma^n}(\sns M,\bm x^n))$$ and apply Shah's Theorem to functions $$\eta_{\bm A}(\sns N,\bm y^1,\dots, \bm y^n)=\chi_{\{h_n(\sns M,\bm x^1,\dots,\bm x^n)=p_n(A^1,\dots,A^n)\}}((1,\bm v^1,\dots,\bm v^n)(\sns N,\bm y^1,\dots, \bm y^n))$$ for all nonnegative integers $A^j$. These functions are not continuous, but are indicators of nice sets.  Using a standard approximation argument we can apply the Theorem to them as well. For $\sns M$ of the form $$\sns M=
\begin{pmatrix}
N^{-1/(d-1)}\\&\ddots\\&&N^{-1/(d-1)}\\&&&N
\end{pmatrix}
\cdot
\begin{pmatrix}
 1\\
&\ddots\\
&&1\\
\alpha_1&\dots&\alpha_{d-1}&1
\end{pmatrix},
$$ we recover the numbers of points in the $n$ segments. In fact, for all $\sns M$ 
$$\eta_{\bm A}((1,\bm v^1,\dots,\bm v^n)^{-1}(\sns M,0,\dots,0)(1,\bm v^1,\dots,\bm v^n))=\chi_{\{h_n(\sns M,\bm x^1,\dots,\bm x^n)=p_n(\bm A)\}}(\sns M,\sns M\bm v^1,\dots,\sns M\bm v^n)$$ and $$g_{\tau,\sigma}(\sns M,\sns M\bm v)=\sum_{\bm m\in\Z^d\setminus\{0\}} \chi_1\dots\chi_1\chi_{(\tau,\tau+\sigma)}(\sns M(\bm m+\bm v)),$$ and for  the particular choice of $\sns M$ above, 
$$\sns M(\bm m+\bm v)=
\begin{pmatrix}
m_1/N^{-1/(d-1)}\\
\vdots\\
m_{d-1}/N^{-1/(d-1)}\\
(m_1\alpha_1+\dots+m_{d-1}\alpha_{d-1}+\xi)N
\end{pmatrix}.
$$
Thus, $\eta_{\bm A}(\phi(\sns M))=1$ if and only if there are exactly $A^1,\dots,A^n$ visits to the segments around $\xi^1,\dots,\xi^n$ for the given $\bm\alpha$ and $N$. 
Hence the form of the limiting distribution is \beq\P_{n,\bm\sigma, \Xi,\bm\tau/\Xi}^{(d)}(A^1,\dots,A^n)=\mu_{\hat L_V/\hat\Lambda_V}\{f_{\bm\tau,\bm\sigma}(\sns M,\bm x^1,\dots,\bm x^n)=(A^1,\dots,A^n)\}.\label{eq:form}\eeq

\end{proof}

\section{Large $d$ Limit}

In this section we consider the large $d$ limit of the distributions from the previous sections and prove Theorem 2. Before proving the theorem, we will need basic information about the Poisson distribution. 

Poisson distribution with parameter $\sigma$ weighs each non-negative integer $k$ with weight $e^{-\sigma}\sigma^k/k!$ We will denote  Poisson distribution with parameter $\sigma$ by $\Pois{\sigma}$. Its moments have the form 
$$\sum_{k=0}^\infty k^n e^{-\sigma}\frac{\sigma^k}{k!}=e^{-\sigma}\left(\sigma\frac{d}{d\sigma}\right)^n e^\sigma=\sum_{k=1}^nS(n,k)\sigma^k,$$ where $S(n,k)$ is the Stirling number of the second kind. As can be easily seen from the above equality, the Stirling number is the number of partitions of a set of $n$ elements into $k$ nonempty sets. The first few moments of the Poisson distribution are $\sigma$, $\sigma^2+\sigma$, $\sigma^3+3\sigma^2+\sigma$. These correspond to partitions $\{1\}$; $\{12\}$, $\{11\}$; $\{123\}$, $\{112\}$, $\{121\}$, $\{211\}$, $\{111\}$.

To further study the limiting distributions we have obtained, we will need the following generalization of a proposition of Marklof from \cite{marklof_2000} which goes back to a theorem of Rogers \cite{rogers_moments_1955}. Let $\Gr(n,l)=O(n)/(O(l)\times O(n-l))$ denote the Grassmannian of $l$-planes in \Rn; we assume that the $l$-planes are embedded in \Rn\ with respect to the standard basis. Let $\Gr(n,l)(\Q)=\{\pi\in\Gr(n,l)\mid\pi\subset\Rn\text{ is defined over }\Q\}=\{\pi\in\Gr(n,l)\mid \pi\cap\Z^n\mbox{ is a lattice in }\pi\}$. For $\pi\in\Gr(n,l)(\Q)$, we write $\covol \pi_\Z$ for the covolume of the lattice $\pi_\Z=\pi\cap \Z^n$ in $\pi$. We also set $G=\sldr$, $\Gamma=\sldz$, and fix $\mu$ to be the Haar probability measure on $G/\Gamma$. 

\begin{theorem}\label{th:moments}
Let $F\colon(\R^d)^{\oplus r}\to\R$ be a bounded piecewise continuous function with compact support. Let $f\colon G/\Gamma\to \R$ be defined by $$f(\sns M)=\sum_{\bm m^1,\dots,\bm m^r\in\Z^d}F(\sns M\bm m^1,\dots,\sns M\bm m^r)$$ with $r<d$ a positive integer. Then, the first moment of $f$ is given by the following expression: 
\beq\label{eq:moments}\int_{G/\Gamma} f(\sns M)d\mu (\sns M)=\sum_{l=0}^r\sum_{\pi\in\Gr(r,l)(\Q)}\int_{x\in\pi'} F(x)\frac{dx}{(\covol \pi_\Z)^d},\eeq 
where $\pi'\in\Gr(rd,ld)(\Q)$ is the image of $\pi$ under the embedding $$(x^1,\dots,x^r)\mapsto(x^1,\dots,x^1,\dots,x^r,\dots,x^r)$$ and the measure $dx$ is  the Lebesgue measure on $\pi'$ that should be interpreted as the delta measure at the origin when $l=0$. 

\end{theorem}

\begin{remark}\label{rem:coordinateplanes}
If in the sum over $\bm m^j$ we omit the terms where any of the $\bm m^j$ are $\bm 0$, then in the sum over $\pi\in\Gr(r,l)(\Q)$ we omit planes that are generated by subsets of the standard basis. This follows from the fact that such subsets of $(\Z^d)^r$ are \sldz-invariant. 
\end{remark}

\begin{lemma}
With notation as in the Theorem, we have $$\int_{G/\Gamma}\sum_{\substack{\bm m^1,\dots,\bm m^r\in\Z^d\\ \text{ linearly indep.}}}F(\sns M\bm m^1,\dots, \sns M\bm m^r)d\mu(\sns M)=\int_{\bm x^j\\\in\R^d}F(\bm x^1,\dots, \bm x^r)d\bm x^1\dots d\bm x^r.$$
\end{lemma}

\begin{proof}
First note that the integral is well-defined since linearly independent sets of vectors are preserved by $\Gamma$. Further renormalize $\mu$ so that $\mu(G/\Gamma)=\prod_{k=2}^d\zeta(k)$ for $d\ge 2$ and write $d\mu(\sns M)/\mu(G/\Gamma)$ in the integral; this normalization will be useful later. Write 
$$\sns M=\begin{pmatrix}
          x_{11}&\dots&x_{1d}\\
\vdots&\ddots&\vdots\\
x_{d1}&\dots&x_{dd}
\end{pmatrix}\in G/\Gamma.$$ Then for $1\le r< d$ we have $$\sns M=\left(\begin{array}{c|c}
\begin{matrix}x_{11}&\dots&x_{1r}\\
\vdots&\ddots&\vdots\\
x_{r1}&\dots&x_{rr}
\end{matrix}
&
0_{r\times(d-r)}\\[0.4em] \hline \\[-1em]
\begin{matrix}x_{r+1,1}&\dots&x_{r+1,r}\\
x_{r+2,1}&\dots&x_{r+2,r}\\
\vdots&\ddots&\vdots\\
x_{d1}&\dots&x_{dr}\end{matrix}&
\begin{matrix}\det^{-1}(x_{ij})_{i,j\le r}&0&\dots&0\\
 0&1\\
\vdots&&\ddots\\
0&&&1
\end{matrix}
\end{array}\right)
\cdot
\left(\begin{array}{c|c}
 \mathrm{Id}_{r\times r}
&\begin{matrix}
z_{11}&\dots&z_{1,d-r}\\
\vdots&\ddots&\vdots\\
z_{r1}&\dots&z_{r,d-r}
\end{matrix}\\ \hline
0_{(d-r)\times r}&\begin{matrix}
y_{11}&\dots&y_{1,d-r}\\
\vdots&\ddots&\vdots\\
y_{d-r,1}&\dots&y_{d-r,d-r}
\end{matrix}
\end{array}\right)
$$ where $(y_{ij})\in \mathrm{SL}(d-r,\R).$ In these coordinates \beq\label{eq:measure}d\mu=\prod_{\substack{i\le d\\ j\le d-r}}dx_{ij}\prod_{\substack{i\le r\\ j\le d-r}}dz_{ij}\cdot \delta(1-\det (y_{ij}))\prod_{i,j\le d-r}dy_{ij}.\eeq The last factor is the Haar measure on $\mathrm{SL}(d-r,\R)$ normalized to $\zeta(2)\dots\zeta(d-r)$ (or simply 1 in case $d-r=1$). 

For $j=1,\dots,r$ let $t^j=\gcd\bm m^j$. Writing $\bm m^j/t^j$ for a column vector, we can find $\sns N\in\sldz$ such that $$\begin{pmatrix}  \dfrac{\bm m^1}{t^1} &\dots&\dfrac{\bm m^r}{t^r}\end{pmatrix}=\sns N \begin{pmatrix}
a_{11}&a_{12}&\cdots&a_{1r}\\
0&a_{22}&\dots&a_{2r}\\
\vdots&\vdots&\ddots&\vdots\\
0&0&\cdots&a_{rr}\\
0&0&\cdots&0\\
\vdots&\vdots&\ddots&\vdots\\
0&0&\cdots&0
\end{pmatrix}=\sns NA.$$ $A$ is a matrix with integer entries uniquely determined by the following conditions: 
\begin{itemize}
 \item $a_{ij}$, $1\le i\le j$ are relatively prime for any fixed $j\in\{1,\dots,r\}$ (in particular, $a_{11}=1$);
\item $0\le a_{1j},\dots,a_{j-1,j} <a_{jj}$. 
\end{itemize}
The first condition is due to relative primality of $\bm m^j/t^j$, and the second comes from applying row operations. Given $a_{11}=1, a_{22},\dots,a_{rr}$, the number of possible matrices $A$ of this form is $$\prod_{j=1}^r\phi_{j-1}(a_{jj}),$$ where $\phi_k$ is the number-theoretic function defined by $$\phi_k(p^\eps)=p^{\eps k}\left(1-\frac1{p^k}\right)$$ for $k\ge 1$ and $\phi_0$ is identically 1; $\phi_1$ is Euler totient function. The function $\phi_k(n)$ counts the number of $k$-tuples $(n_1,\dots,n_k)\in\{0,\dots,n-1\}^k$ such that $\gcd(n, n_1, \dots, n_k)=1$. 

Let us compute the stabilizer of a fixed matrix $A$: 
$$\Gamma_A=\{ \gamma\in\Gamma\mid \gamma A=A\}=\left(\begin{array}{c|c} \mathrm{Id}_{r\times r}&\Z_{r\times(d-r)}\\[-0.7em] \\ \hline \\[-0.7em] 0_{(d-r)\times r}& \mathrm{SL}(d-r,\Z)\end{array}\right).$$ Thus we get 
\begin{multline}\frac{1}{\mu(G/\Gamma)}\int_{G/\Gamma}\sum_{\bm m^j\text{ l. i.}} F(\sns M\bm m^1,\dots, \sns M\bm m^r)d\mu(\sns M)=
\\ \frac1{\mu(G/\Gamma)}\sum_{t^1,\dots,t^r=1}^\infty \int_{G/\Gamma}\sum_{\sns N\in \Gamma/\Gamma_A}F(\sns M\sns N\begin{pmatrix}a_{11}\\0\\\vdots \\ 0\end{pmatrix}, \sns M\sns N\begin{pmatrix}a_{12}\\a_{22}\\\vdots \\ 0\end{pmatrix},\dots,\sns M\sns N\begin{pmatrix}a_{1r}\\\vdots \\ a_{rr}\\\vdots\end{pmatrix})d\mu(\sns M)=
\\\frac1{\mu(G/\Gamma)}\sum_{t^j}\int_{G/\Gamma_A}F(\sns M\begin{pmatrix}a_{11}\\0\\\vdots \\ 0\end{pmatrix}, \sns M\begin{pmatrix}a_{12}\\a_{22}\\\vdots \\ 0\end{pmatrix},\dots,\sns M\begin{pmatrix}a_{1r}\\\vdots \\ a_{rr}\\\vdots\end{pmatrix})d\mu(\sns M).\end{multline} Using \eqref{eq:measure} to change the measure we get $$\tfrac1{\zeta(d)\dots\zeta(d-r+1)}\sum_{t^j}\sum_A\!\int_{(\R^d)^r}\!\!\!F(t^1 \bm x^1, t^2a_{12}\bm x^1+t^2 a_{22} \bm x^2,\dots,t^ra_{1r}\bm x^1+\dots+ t^r a_{rr}\bm x^r)d\bm x^1\dots d\bm x^r.$$ Now we do a linear change of variables and get 
\beq\tfrac{1}{\zeta(d)\dots\zeta(d-r+1)}\sum_{t^j=1}^\infty\tfrac1{(t^1\dots t^r)^d}\!\!\sum_{a_{22},\dots,a_{rr}=1}^\infty\!\!\!\! \tfrac{\phi_1(a_{22})}{a_{22}^d}\dots \tfrac{\phi_{r-1}(a_{rr})}{a_{rr}^d}\cdot\!\!\int_{(\R^d)^r}\!\!\! F(\bm x^1,\dots, \bm x^r) d\bm x^1\dots d\bm x^r.\label{eq:phi}\eeq
It is easy to see that $$\sum_{n\ge 1}\frac{\phi_k(n)}{n^d}=\frac{\zeta(d-k)}{\zeta(d)},$$ whence the constant in front of the integral in \eqref{eq:phi} is 1, as desired.

\end{proof}

\begin{proof}[Proof of Theorem \ref{th:moments}] Rewrite the integral we are evaluating as 
$$\frac1{\mu(G/\Gamma)}\int_{G/\Gamma}\sum_{l=0}^r\sum_{\rk (\bm m^1, \dots, \bm m^r) = l}F(\sns M\bm m^1,\dots, \sns M\bm m^r)d\mu(\sns M).$$ Here we normalize $\mu$ like in the Lemma and the inner sum runs over those $r$-tuples of vectors whose $\R$-span has dimension $l$. Since the set $\{\rk \bm m=l\}$ is \sldz-invariant for each $l$, we can pass the sum over $l$ through the integral sign. To prove the Theorem, it suffices to show that corresponding terms in the expression above and in \eqref{eq:moments} match for each $l$. Now observe that $$\{\rk \bm m=l\}=\bigcup_{\pi\in\Gr (l,r)(\Q)} \{\mbox{$r$-tuples of vectors from $(\pi\cap \Z^r)^d$ with rank $l$}\}.$$ In fact, the sets whose union we are taking constitute an \sldz-invariant partition. Thus we need to parametrize linearly independent vectors in $(\pi\cap \Z^r)^d$ for each $\pi.$ Let $B\colon \R^l\to\R^r$ be a linear map with image $\pi$ such that $B(\Z^l)=\pi\cap \Z^r.$ There can be many of these; any one will do. Using the standard basis, $B=(b_i^j)$ with $1\le i\le l$ and $1\le j\le r$, and we obtain $\bm m^j=\sum_i b_i^j\bm n^i$, where $\bm n^i\in\Z^{d}$ form a linearly independent set. Thus the integral becomes $$\frac1{\mu(G/\Gamma)}\int_{G/\Gamma}\sum_{\pi\in\Gr (l,r)(\Q)}\left[\sum_{\substack{\bm n^1,\dots,\bm n^l\in\Z^d\\\text{linearly indep.}}}F(\sns M \sum b_i^1\bm n^i,\dots,\sns M \sum b_i^r \bm n^i)\right]d\mu(\sns M).$$ The quantity in brackets is $\Gamma$-invariant, so the sum over $\pi$ can be interchanged with the integral. 

For each $\pi$ and $B$ we can now apply the Lemma. It gives $$\sum_{\pi\in\Gr (l,r)(\Q)}\int_{(\R^d)^r}F(\sum b_i^1\bm x^i,\dots, \sum b_i^r \bm x^i)d\bm x^1\dots d\bm x^l.$$ Since $B(\Z^l)=\pi\cap \Z^r,$ the Jacobian of $B$ is the covolume of $\pi$. The statement of the Theorem follows after a linear change of variables. 
 
\end{proof}

% For any collection of vectors $\bm m^1,\dots, \bm m^r$ there exists $\sns N\in\sldz$ such that 
% $$
% \begin{pmatrix}
% m_1^1&\cdots&m_1^r\\
% \vdots&\ddots&\vdots\\
% m_d^1&\cdots&m_d^r\\
% \end{pmatrix}=
% \sns N
% \begin{pmatrix}
% a_{1}^1&a_{1}^2&\cdots&a_{1}^r\\
% 0&a_{2}^2&\dots&a_{2}^r\\
% \vdots&\vdots&\ddots&\vdots\\
% 0&0&\cdots&a_{r}^r\\
% 0&0&\cdots&0\\
% \vdots&\vdots&\ddots&\vdots\\
% 0&0&\cdots&0
% \end{pmatrix}=\sns NA
% $$ for a $d\times r$ matrix $A$ with integer entries. Then 
% $$\int f(\sns M)d\mu(\sns M)=\sum_{A\in\Z^{rd}\text{ as above}}\int_{\sldr/\Stab A} F(\sns M\bm a^1,\dots,\sns M\bm a^r)d\mu(\sns M)$$ where $\bm a^l$ are the column vectors. Ordering by the rank of $A$ and writing $\bm x^l$ for the column vectors of $\sns M$ we get 
% $$\sum_{l=0}^r \sum_{\substack{A\in\Z^{rd}\text{ as above}\\\rk A=l}}\int_{\bm x^l\in \R^d}F(\bm x^1a_{1}^1,\bm x^1a_{1}^2+\bm x^2a_{2}^2,\dots,\bm x^1a_{1}^r+\dots+\bm x^ra_{r}^r)d\bm x^1\dots d\bm x^r.$$ It is clear that each $A$ corresponds to a plane $\pi\in\Gr(r,l)(\Q)$ and that the Jacobian of $A$ is $(\covol\pi_\Z)^d$

\begin{prop}
For any $\Xi$ and $\tau$ we have $$\P^{(d)}_{1,\sigma,\Xi,\tau/\Xi}\Longrightarrow \Pois\sigma$$ as $d\to\infty$. 
\end{prop}

\begin{proof}From \eqref{eq:form}, all we need to show is that moments of $f_{\tau,\sigma}$ are Poissonian for large $d$. First consider that case when $\xi \not\in\Q$.  Without loss of generality we set $\tau=0$. Then we need to find 
\beq\label{eq:poisson}\lim_{d\to\infty}\int_{\sns M\in G/\Gamma}\int_{\bm v} \left[\sum_{\bm m\in\Z^d\setminus\{0\}} (\chi_1\cdot\ldots\cdot\chi_1\cdot\chi_\sigma)(\sns M\bm m+\bm v)\right]^kd\mu(\sns M)d\bm v\eeq
for $k=0,1,\dots$ Taking the integral over $\bm v$ inside the sum, we clear the way for Theorem \ref{th:moments} applied to $$F(\bm x_1,\dots,\bm x_d)=G_1(\bm x_1)\dots G_1(\bm x_{d-1})G_\sigma(\bm x_d)$$ where $$G_t(\bm z)=\int_{y=0}^1\chi_t(z_1+y)\dots\chi_t(z_k+y)dy.$$ For any plane $\pi'$ as in the Theorem, we have that 
\beq\label{eq:covol}\int_{\pi'}F(x)dx/(\covol \pi)^d=\left(\frac{\int_\pi G_1(\bm x_1)d\bm x_1}{\covol \pi}\right)^d\cdot\frac{\int G_\sigma(\bm x_d)d\bm x_d}{\int G_1(\bm x_d)d\bm x_d}.\eeq
It is elementary to see that the quantity raised to the power $d$ is at most one: the numerator is the volume of $\pi$ ``lying'' inside the ``crystal'' shape, and $\covol\pi_\Z$ is the volume of a fundamental domain. To wit, consider first the case when $G_t(\bm z)$ is replaced by the indicator of $[0,1]^k$. Since vertices of any fundamental domain for $\pi_\Z$ have integer coordinates, it can completely cover the part of the plane inside the cube. Furthermore, the quantity in parentheses can equal one only when $\pi\cap [0,1]^k$ constitutes a fundamental domain for $\pi_\Z$. This means that there exists a \Z-basis $\{e_i\}_{1}^k$ for $\pi_\Z$ such that 
\begin{itemize}\item $e_i\in\{0,1\}^k$, $1\le i\le k$; \item $e_i+e_{i'}\in \{0,1\}^k$, $1\le i, i'\le k$. 
\end{itemize}
Hence two distinct $e_i$, $e_{i'}$ cannot take on the value 1 in the same coordinate. The same argument extends to other cubes of the form $[-s,1-s]^k$ for $s\in [0,1]$ and so, too, for the original $G_t(\bm z)$ as it is an average over cubes of this kind. 

The above argument shows that the limit as $d\to\infty$ exists for each moment and that rate of convergence is exponential. To understand this limit, we focus on the terms with $\int_\pi G(x_1)dx_1/\covol\pi=1$. Since in \eqref{eq:poisson} we omit the terms in which any of $m^l=0$, the only terms that survive after taking the limit are the ones with $\pi$ generated by $e_i$ for which $\sum_{i=1}^k e_i=(1,\dots,1)$ (no zero coordinates). For planes $\pi$ of  fixed dimension $l$ the number of possibilities is the number of partitions of a set of $k$ elements into $l$ non-empty subsets, which is exactly $S(k,l)$. Finally observing that the last factor in \eqref{eq:covol} is $\sigma^{\dim \pi}=\sigma^l$, we find that the $k$-th moment tends to $$\sum_{l=1}^k S(k,l)\sigma^l,$$ which is the corresponding moment of the Poisson distribution with parameter $\sigma.$

In the case when $\xi=p/q\in\Q$ we can modify the above proof. The integral over $v$ becomes a finite sum, and we let $G_t(z)=\frac1q\sum_{r=0}^{q-1}\chi_t(z_1+r/q)\dots\chi_t(z_d+r/q)$; the variable $\tau$ appears in an equation similar to \eqref{eq:covol} and doesn't enter the definition of $G_t(z)$. The statements from the continuous version are true for this function as well (since it is also an average over cubes), and the proof is complete.

\end{proof}

Generalizing this proposition we can obtain the statement of Theorem 2. 

\begin{proof}[Proof of Theorem 2]
What we need to show is that $$\int_{\sns M}\int_{V\in\Xi^d}
\sum_{\substack {\bm m^{1,1}, \dots, \bm m^{k^1,1}\\ \vdots \\ \bm m^{1,n},\dots,\bm m^{k^n,n}}\in\Z^d\setminus\{0\}} 
\prod_{j=1}^n\prod_{j'=1}^{k^j} \chi_1\cdot\ldots\cdot\chi_1\cdot\chi_{(\tau^j,\sigma^j+\tau^j)}(\sns M\bm m^{j',j}+\bm v^{j})\,d\mu(\sns M)\,dV$$ has a limit as $d\to\infty$ for every choice of $k^1,\dots,k^j$. If $\bm Y^{(d)}_{n, \bm\sigma,\Xi,\tau/\Xi}=(Y^1,\dots,Y^n)$ is distributed according to $\P^{(d)}_{n,\bm\sigma,\Xi,\bm\tau/\Xi}$, then this expression is nothing more than the moment of order $(k^1,\dots, k^j)$. 

Now we make a simplifying observation: we can assume that $k^j=1$ for all $j$ without loss of generality since taking all possible $n$ and $\Xi$ and computing $\E\prod Y^j$ produces all the moments $\E\prod(Y^j)^{k^j}$. That is, duplicating the random variable $Y^j$ $k^j$ times allows us to assume that $k^j$ is 1. So we need to analyze 
$$\int_{\sns M}\int_{V\in\Xi^d}
\sum_{\bm m^j\in\Z^d\setminus\{0\}} 
\prod_{j=1}^n \chi_1\cdot\ldots\cdot\chi_1\cdot\chi_{(\tau^j,\sigma^j+\tau^j)}(\sns M\bm m^{j}+\bm v^{j})\,d\mu(\sns M)\,dV,$$ which by Theorem \ref{th:moments} is 
\beq\mathop{\sum^\prime\nolimits}_{\pi\in\Gr(r,l)(\Q)}\int_{V\in \Xi^d}\int_{\pi'}\label{eq:generalmoment}\frac{dx}{(\covol\pi_\Z)^d}\prod_{j=1}^{n}\chi_1\dots\chi_{(\tau^j,\tau^j+\sigma^j)}(x^{j}+\bm v^j)dV.\eeq 
Since $\bm m^j$ are non-zero, we exclude the ``coordinate planes'' as in Remark \ref{rem:coordinateplanes}; this is denoted by the prime in the formula above. 

We need to account for planes $\pi\in\Gr(r,l)(\Q)$ that will contribute in the limit $d\to\infty.$ By the argument from the previous proposition we have that \beq\label{eq:cube}\int_{\pi'}\prod_{j=1}^{n}\chi_1\dots\chi_{(\tau^j,\tau^j+\sigma^j)}(x^{j}+\bm v^j)dx\le (\covol\pi_\Z)^d.\eeq Since $\Xi^d$ is normalized to have measure 1, it suffices to study the integrand for fixed $V\in\Xi^d$. If we can find $V$ and $\pi$ for which strict inequality is true in \eqref{eq:cube}, then by continuity we have strict inequality for the integral over $V\in\Xi^d$ and thus conlude that $\pi$ doesn't contribute in the limit. We will do this for $V=0$ first. A plane that will contribute in the limit $d\to\infty$ must satisfy the property that $\pi\cap[0,1]^r$ is a fundamental domain for $\pi\cap\Z^r$ as in the previous proposition. For each of these planes we can try to find another $V$ that gives strict inequality in \eqref{eq:cube}. If $V\in\pi$, we are translating the cube along the plane and thus getting the same cross-sectional area. So suppose $V\in\Xi\setminus\pi$; this corresponds to cutting the cube with a plane parallel to $\pi$. It is easy to see that for such planes the section will always have smaller area than the one for $V\in\pi$. Thus it must be the case that $\Xi\subset\pi$. 

To summarize, a plane $\pi$ contributes to the limit only if $\pi\cap [0,1]^r$ is a fundamental domain for $\pi\cap\Z^r$ and $\Xi\subset\pi$. This means that $\P^{(d)}_{n,\bm\sigma,\Xi,\bm\tau/\Xi}$ has a limit as $d\to\infty$ because all moments exist. If we write $\bm\xi=(\xi^1,\dots,\xi^1,\xi^2,\dots,\xi^2,\dots,\xi^{n'},\dots,\xi^{n'})$ reordering as necessary, then $\pi$ must a product of admissible planes for $(\xi^1,\dots,\xi^1)$, $(\xi^2,\dots,\xi^2)$, \dots, $(\xi^{n'},\dots,\xi^{n'})$. Hence the moment will split as the product of  moments over distinct $\xi^j$. Using this observation and  the previous proposition we see that in the case of distinct $\xi^j$ in Theorem 2 the limiting distribution is the product of independent Poisson distributions. If $\xi^j=\xi^{j'}$ but $(\tau^j,\tau^j+\sigma^j)\cap(\tau^{j'},\tau^{j'}+\sigma^{j'})=\varnothing$, then such $\xi^j$ and $\xi^{j'}$ behave as if they were unequal since the factor $$\prod \chi_{(\tau^j,\tau^j+\sigma^j)}(x^j+\bm v^j)$$ from \eqref{eq:cube} vanishes. It is evident that if $(\tau^j,\tau^j+\sigma^j)\cap(\tau^{j'},\tau^{j'}+\sigma^{j'})\ne\varnothing$ for some $j,j'$, then the limiting distribution cannot be a product of independent distributions. This concludes the proof of Theorem 2. 
\end{proof}

\bibliographystyle{plain}

\bibliography{bibliography}

\end{document}